\documentclass[11pt,twoside]{amsart}
\usepackage[latin1]{inputenc}

\usepackage{amsmath}
\usepackage{amsfonts}
\usepackage{amssymb}
\usepackage{amsthm}
\usepackage{color}
\usepackage[usenames]{xcolor}
\usepackage{graphicx}
\usepackage[all]{xy}
\usepackage{fullpage}
\usepackage{hyperref}
\usepackage{multicol}
\usepackage{listings}
\usepackage{verbatim}
\usepackage{endnotes}
\usepackage{enumitem}
\usepackage[abs]{overpic}

\theoremstyle{definition}
\newtheorem{defn}{Definition}[section]

\theoremstyle{plain}
\newtheorem{thm}[defn]{Theorem}
\newtheorem{lem}[defn]{Lemma}
\newtheorem{prop}[defn]{Proposition}
\newtheorem{cor}[defn]{Corollary}

\def\C{\ensuremath{\mathbb{C}}}
\def\G{\ensuremath{\mathbb{G}}}

\def\P{\ensuremath{\mathbb{P}}}
\def\Q{\ensuremath{\mathbb{Q}}}
\def\R{\ensuremath{\mathbb{R}}}
\def\Z{\ensuremath{\mathbb{Z}}}

\def\AA{\ensuremath{\mathcal A}}

\def\FF{\ensuremath{\mathcal F}}

\def\II{\ensuremath{\mathcal I}}

\def\OO{\ensuremath{\mathcal O}}

\def\TT{\ensuremath{\mathcal T}}

\def\ch{\mathop{\mathrm{ch}}\nolimits}

\def\Coh{\mathop{\mathrm{Coh}}\nolimits}

\def\deg{\mathop{\mathrm{deg}}\nolimits}

\def\dim{\mathop{\mathrm{dim}}\nolimits}

\def\Eff{\mathop{\mathrm{Eff}}\nolimits}

\def\ext{\mathop{\mathrm{ext}}\nolimits}
\def\Ext{\mathop{\mathrm{Ext}}\nolimits}

\def\Hilb{\mathop{\mathrm{Hilb}}\nolimits}

\def\Hom{\mathop{\mathrm{Hom}}\nolimits}

\def\id{\mathop{\mathrm{id}}\nolimits}

\def\PGL{\mathop{\mathrm{PGL}}}

\def\Pic{\mathop{\mathrm{Pic}}\nolimits}

\def\tilt{\mathop{\mathrm{tilt}}}

\def\Stab{\mathop{\mathrm{Stab}}\nolimits}

\def\into{\ensuremath{\hookrightarrow}}
\def\onto{\ensuremath{\twoheadrightarrow}}

\begin{document}

\title{Families of elliptic curves in $\P^3$ and Bridgeland Stability}

\author{Patricio Gallardo}
\address{Department of Mathematics, University of Georgia, Athens, GA 30602, USA}
\email{gallardo@math.uga.edu}
\urladdr{https://sites.google.com/site/patriciogallardomath/}

\author{C\'esar Lozano Huerta}
\address{Instituto de Matem\'aticas, CONACYT-UNAM, Oaxaca de Ju\'arez, Oax. 6800, M\'exico }
\email{lozano@im.unam.mx}
\urladdr{http://www.matem.unam.mx/lozano}

\author{Benjamin Schmidt}
\address{Department of Mathematics, The University of Texas at Austin, 2515 Speedway, RLM 8.100, Austin, TX 78712, USA}
\email{schmidt@math.utexas.edu}
\urladdr{https://sites.google.com/site/benjaminschmidtmath/}



\begin{abstract}
We study wall crossings in Bridgeland stability for the Hilbert scheme of elliptic quartic curves in three dimensional projective space. We provide a geometric description of each of the moduli spaces we encounter, including when the second component of this Hilbert scheme appears. Along the way, we prove that the principal component of this Hilbert scheme is a double blow up with smooth centers of a Grassmannian, exhibiting a completely different proof of this known result by Avritzer and Vainsencher. This description allows us to compute the cone of effective divisors of this component.
\end{abstract}

\maketitle
\section*{introduction}
\label{sec:introduction}

The global geometry of a given Hilbert scheme is generally very difficult to study. Recently, the theory of Bridgeland stability has provided a new set of tools to study the geometry of these Hilbert schemes. For instance, the study of the Hilbert scheme of points on surfaces has benefited from these new tools (see \cite{ABCH13, BM14, CHW17, LZ16, MM13, Nue16, YY14}). A sensible step forward is now to apply these tools to examine families of curves contained in threefolds. The first instance of this was carried out by the last author in \cite{Sch15}, where he studies the Hilbert scheme of twisted cubics. This paper continues this investigation about curves in $\P^3$ and analyzes the global geometry, as well as wall-crossing phenomena, of the Hilbert scheme $\Hilb^{4t}(\P^3)$, which parametrizes subschemes of $\P^3$ of genus $1$ and degree $4$.

A smooth curve of genus $1$ and degree $4$ in $\P^3$, which we refer to as an elliptic quartic, is the transversal intersection of two quadric surfaces. By considering the pencil that these quadrics generate, we realize the family of smooth elliptic quartics as an open subset of $\G(1,9)$, the Grassmannian  of lines in the space $ |\mathcal{O}_{\P^3}(2)|$ of quadric surfaces in $\P^3$. We show that the Hilbert scheme $\Hilb^{4t}(\P^3)$ is a moduli space of Bridgeland stable objects, and moreover, one of its components is related through birational transformations to the Grassmannian $\G(1,9)$ via wall-crossing.

Let us recall the notion of Bridgeland stability in order to state this result precisely. For classical slope stability with respect to a given polarization $H$ on a smooth projective complex variety $X$, one defines a number
$\mu_H(E) = \tfrac{H^{n-1} \cdot \ch_1(E)}{H^n \cdot \ch_0(E)}$
called the slope for any coherent sheaf $E \in \Coh(X)$. A coherent sheaf is then called slope semistable if all proper non trivial subsheaves have smaller slope. For Bridgeland stability, one replaces the category of coherent sheaves with a different abelian subcategory $\AA \subset D^b(X)$ and replaces the slope with a homomorphism $Z: K_0(X) \to \C$, mapping $\AA$ to the upper half plane or the negative real line, where $K_0(X)$ is the Grothendieck group. The slope is then given by
\[
\mu(E) = -\frac{\Re{Z(E)}}{\Im{Z(E)}}
\]
for any $E \in \AA$. In addition, one demands that every object in $D^b(X)$ has a canonical filtration into semistable factors called the Harder-Narasimhan filtration and the so called support property, which ensures that the set of stability conditions $\Stab(X)$ can be naturally given the structure of a complex manifold.

We can now state our main result. Let us fix a class $v \in K_0(X)$, then there is a locally finite wall and chamber structure in $\Stab(X)$, such that the set of semistable objects of class $v$ is constant within each chamber. Our main result describes the wall and chamber structure of a subspace of $\Stab(\P^3)$ as well as the corresponding moduli spaces of semistable objects in the case of elliptic quartics in $\P^3$.

\newtheorem*{thm:wall_crossings}{Theorem A}
\begin{thm:wall_crossings}
Let $v = (1, 0, -4, 8) = \ch(\II_C)$, where $C \subset \P^3$ is an elliptic quartic curve. There is a path $\gamma: [0,1] \to \R_{>0} \times \R \subset \Stab(\P^3)$ such that the moduli spaces of semistable objects with Chern character $v$ in its image outside of walls are given in the following order.
\begin{enumerate}
\setcounter{enumi}{-1}
  \item The empty space $M_0 = \emptyset$.
  \item The Grassmannian $M_1 = \G(1,9)$ parametrizing pencils of quadrics. The only non-ideal sheaves in the moduli space come from the case, where a $2$-plane is contained in the base locus of the pencil.
  \item The second moduli space $M_2$ is the blow up of $\G(1,9)$ along a smooth locus isomorphic to $\G(1,3) \times (\P^3)^{\vee}$ parametrizing the non-ideal sheaves in $M_1$. The exceptional divisor generically parametrizes unions of a line and a plane cubic intersecting themselves in a single point. The only non-ideal sheaves in this moduli space come from the case when the line is contained in the plane.
  \item The third moduli space $M_3$ has two irreducible components $M_3^1$ and $M_3^2$. The first component $M_3^1$ is the blow up of $M_2$ along the smooth incidence variety parametrizing length two subschemes in a plane in $\P^3$. The second component $M_3^2$ is a $\P^{14}$-bundle over $\Hilb^2(\P^3) \times (\P^3)^{\vee}$. It generically parametrizes unions of plane quartics with two points, either outside the curve or embedded. The two components intersect transversally along the exceptional locus of the blow up. The only non-ideal sheaves occur in the case where at least one of the two points is not scheme-theoretically contained in the plane.
  \item The fourth moduli space $M_4$ has two irreducible components $M_4^1$ and $M_4^2$. The first component is equal to $M_3^1$. The second component is birational to $M_3^2$. The new locus parametrizes plane quartics with two points, such that exactly one point is scheme-theoretically contained in the plane.
  \item The fifth moduli space is the Hilbert scheme $\Hilb^{4t}(\P^3)$, which has two components: $\Hilb_1^{4t}$ and $\Hilb_2^{4t}$. The principal component $\Hilb_1^{4t}$ contains an open subset of elliptic quartic curves and is equal to $M_3^1$. The second component is of dimension $23$ and is birational to $M_3^2$. Moreover, the two components intersect transversally along a locus of dimension $15$. The component $\Hilb_2^{4t}$ differs from $M_4^2$ in the locus of plane cubics together with two points scheme-theoretically contained in the plane. 
\end{enumerate}
\end{thm:wall_crossings}

\begin{figure}
\centering
\begin{overpic}[scale=0.75,unit=1 mm]{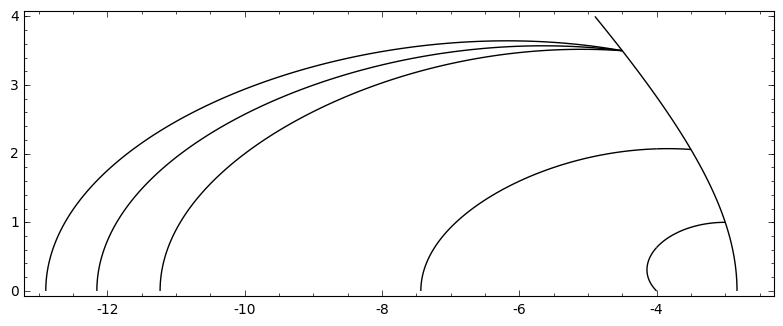} 
\put(12, 50) {$\Hilb^{4t}(\P^3)$}
\put(12, 12) {$M_4$}
\put(23, 12) {$M_3$}
\put(60, 12) {$M_2$}
\put(95, 12) {$\G(1,9)$}
\put(130, 12) {$\emptyset$}
\end{overpic}
\caption{Wall and chamber structure in a subspace of $\Stab(\P^3)$ for $\Hilb^{4t}(\P^3)$ and their associated models according to Theorem A}
\label{fig:bridgeland_walls_elliptic}
\end{figure}
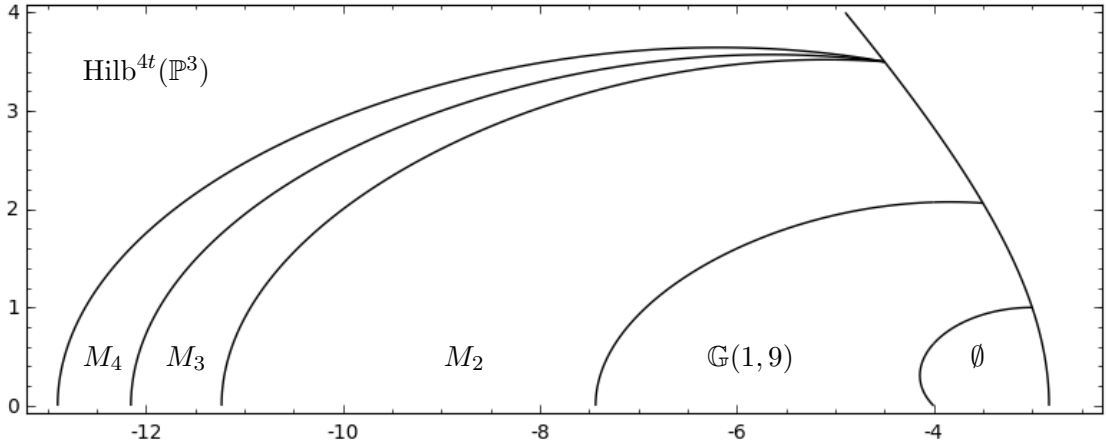

As a consequence of Theorem A, we obtain that the Hilbert scheme $\Hilb^{4t}(\P^3)$ has two components. This is a well known fact (see \cite{CN12, Got08}). More interestingly, the previous result describes what is called the principal component, which parametrizes smooth elliptic curves along with their flat limits. We will denote this component by $\Hilb_1^{4t}$, and our next result describes its global geometry. 

\newtheorem*{thm:va92}{Theorem B}
\begin{thm:va92}[\cite{VA92}]
The closure of the family of smooth elliptic quartics in the Hilbert scheme $\Hilb^{4t}(\P^3)$, is a double blowup of the Grassmannian $\G(1,9)$ along smooth centers.
\end{thm:va92}

A comment is in order about the previous theorem. The description of $\Hilb_1^{4t}(\P^3)$ above was proved in \cite{VA92} by Vainsencher and Avritzer using classical methods. Our techniques to reprove their result are distinct, as we make use of the bounded derived category of coherent sheaves on $\P^3$ and Bridgeland stability.

Since the principal component $\Hilb_1^{4t}(\P^3)$ is a double blowup, it is natural to ask what are the subschemes of $\P^3$ that the exceptional divisors parametrize and whether they span extremal rays in the cone of effective divisors $\Eff(\Hilb_1^{4t})$. Proposition \ref{prop:E1_E2_description}, and the following result answer these two questions. Consequently, we have a moduli interpretation for the generators of $\Eff(\Hilb_1^{4t})$, which is the following.

Let $E_1$ be the closure of the locus parametrizing subschemes of $\P^3$ that are the union of a plane cubic and an incident line. By $E_2$ we denote the closure of the locus parametrizing plane quartics with two nodes and two embedded points at such nodes. Let $\Delta$ denote the closure of the locus of nodal elliptic curves. 

\newtheorem*{thm:cones}{Theorem C}
\begin{thm:cones}
The cone of effective divisors of $\Hilb_1^{4t}$ is generated by $\Eff(\Hilb_1^{4t})=\langle E_1, E_2, \Delta\rangle$. 
\end{thm:cones}

\subsection*{Ingredients}

The notion of tilt stability on a smooth projective threefold was introduced in \cite{BMT14}. It is defined in a similar way one defines Bridgeland stability on a surface. Thus, these two notions of stability share computational properties. Tilt stability is intended as a stepping stone to Bridgeland stability. The proof of Theorem A is mostly based on this theory.

In contrast to the surface case, computing which objects destabilize at a given wall is difficult due to the lack of unique stable factors in the Jordan-H\"older filtration of a strictly semistable object. Computing the walls numerically in tilt stability is of similar difficulty as in the surface case and often times possible. On the other hand, while it is generally difficult to determine all walls in Bridgeland stability on a given path, it is not so difficult to determine which objects destabilize at a given wall. In order to resolve this issue we apply a technique from \cite{Sch15} that allows to translate walls from tilt stability into Bridgeland stability.

In order to identify the global structure of the Bridgeland moduli spaces, a careful analysis of its singularities is necessary. We apply deformation theory to these problems, and large parts of it reduce to heavy $\Ext$-computation. Even though this can be done by hand, computer calculations with \cite{M2} turn out to be tremendously helpful. The situation is more involved when it comes to the intersection of the two components. We reduce the question to a single ideal in that case and apply the technique of \cite{PS85}. We make use of the Macaulay2 implementation \cite{Ilt12} of this technique.

The proof of Theorem C uses the description of the exceptional divisors provided in Proposition \ref{prop:E1_E2_description}, and exhibits the dual curves to them in order to conclude.

\subsection*{Organization}
In Section 1, we recall basic definitions about stability conditions. In Section 2, we carry out numerical computations in tilt stability needed to understand walls in Bridgeland stability. In Section 3, we describe the equations of some ideals depending on the exact sequences they fit in. We use this description to understand the local geometry of the intersection of the two components of our Hilbert scheme. In Section 4, we translate the computations in tilt stability to Bridgeland stability. Furthermore, we analyze singularities to provide a proof of Theorem A and Theorem B. In Section 5, we prove Theorem C. Appendix A contains our Macaulay2 code.

\subsection*{Acknowledgements}
We would like to thank Francesco Cavazzani, Dawei Chen, Izzet Coskun, Joe Harris, Sean Keel, Emanuele Macr\`i, and Edoardo Sernesi for insightful discussions about this work. We are also grateful to the referee for carefully going through the article. C. Lozano Huerta and B. Schmidt would also like to thank the organizers of the II ELGA school in Cabo Frio, Brazil for organizing a wonderful conference supported by NSF grant DMS-1502154 at which parts of this work was done. P. Gallardo is supported by NSF grant DMS-1344994 of the RTG in Algebra, Algebraic Geometry, and Number Theory, at the University of Georgia. C. Lozano Huerta thanks the Department of Mathematics at Harvard University 	for providing ideal working conditions. He is funded by CONACYT Grant CB-2015-01 No. 253061. B. Schmidt has been partially supported by NSF grant DMS-1523496 (PI Emanuele Macr\`i) and a Presidential Fellowship of the Ohio State University. He also wants to thank Northeastern University, where part of this work was done, for their hospitality

\subsection*{Notation}
We work over the field of the complex numbers throughout. We also use the following notation.
\begin{center}
  \begin{tabular}{ r l }
    $\II_{Z/X}$, $\II_Z$ & ideal sheaf of a closed subscheme $Z \subset X$ \\
    $D^b(X)$ & bounded derived category of coherent sheaves on $X$ \\
    $\ch_X(E)$, $\ch(E)$ & Chern character of an object $E \in D^b(X)$  \\
    $\ch_{\leq l,X}(E)$, $\ch_{\leq l}(E)$ & $(\ch_{0,X}(E), \ldots,
    \ch_{l,X}(E))$ \\
    $\G(r,k)$ & the Grassmannian parametrizing subspaces $\P^r \subset \P^k$ \\
    $\Hilb_1^{4t}$ & closure of the locus of elliptic quartic curves in $\Hilb(\P^3)$ \\
    $\Hilb_2^{4t}$ & closure in $\Hilb(\P^3)$ of the locus of unions of plane  \\
    & quartic curves with two points in $\P^3$ \\
  \end{tabular}
\end{center}

\section{Preliminaries on Stability Conditions}
\label{sec:prelim}

This section recalls the construction of Bridgeland stability conditions on $\P^3$ due to \cite{BMT14, MacE14}. We refer the reader to \cite{Bri07} for a detailed introduction to the theory of Bridgeland stability. Let $X$ be a smooth projective threefold. A \textit{Bridgeland stability condition} on $D^b(X)$ is a pair $(Z,\AA)$, where $\AA$ is the heart of a bounded t-structure and $Z: K_0(X) = K_0(\AA) \to \C$ is an additive homomorphism that maps any non trivial object in $\AA$ to the upper half-plane or the negative real line. Additionally, technical properties such as the existence of Harder-Narasimhan filtrations and the support property have to be fulfilled. Bridgeland's main result is that the set of stability condition can be given the structure of a complex manifold. We will denote this \textit{stability manifold} by $\Stab(X)$.

Let $H$ be the very ample generator of $\Pic(\P^3)$. Due to the simplicity of the cohomology of $\P^3$, we will abuse notation by writing $\ch_i(E) = H^{3-i} \ch_i(E)$ for any $E \in D^b(X)$. If $\beta \in \R$, we define the \textit{twisted Chern character} by $\ch^{\beta} := e^{-\beta H} \cdot \ch$. In more detail, we have
\begin{align*}
\ch^{\beta}_0 &= \ch_0, \, \, \ch^{\beta}_1 = \ch_1 - \beta \ch_0, \, \, \ch^{\beta}_2 = \ch_2 - \beta \ch_1 + \frac{\beta^2}{2} \ch_0, \\
\ch^{\beta}_3 &= \ch_3 - \beta \ch_2 + \frac{\beta^2}{2} \ch_1 - \frac{\beta^3}{6} \ch_0.
\end{align*}
We write a twisted version of the classical \textit{slope function} as
\[
\mu_{\beta}(\ch_0, \ch_1) := \frac{\ch^{\beta}_1}{\ch^{\beta}_0} = \frac{\ch_1}{\ch_0} - \beta,
\]
where division by $0$ is interpreted as $+\infty$. In \cite{BMT14} the notion of \textit{tilt stability} has been introduced as an auxilliary notion in between classical slope stability and Bridgeland stability on threefolds. We will recall this construction and a few of its properties. Tilting is used to obtain a new heart of a bounded t-structure. For more information on the general theory of tilting we refer to \cite{HRS96}. A torsion pair is defined by
\begin{align*}
\TT_{\beta} &:= \{E \in \Coh(\P^3) : \text{any quotient $E \onto G$ satisfies $\mu_{\beta}(G) > 0$} \}, \\
\FF_{\beta} &:=  \{E \in \Coh(\P^3) : \text{any subsheaf $F \subset E$ satisfies $\mu_{\beta}(F) \leq 0$} \}.
\end{align*}
A new heart of a bounded t-structure is given by the extension closure $\Coh^{\beta}(\P^3) := \langle \FF_{\beta}[1], \TT_{\beta} \rangle$. Equivalently, the objects in $\Coh^{\beta}(\P^3)$ are complexes $E \in D^b(X)$ satisfying $H^i(E) = 0$ for $i \neq 0,-1$, $H^{-1}(E) \in \FF_{\beta}$ and $H^0(E) \in \TT_{\beta}$. Let $\alpha > 0$ be a positive real number. The new slope function is
\[
\nu_{\alpha, \beta}(\ch_0, \ch_1, \ch_2) := \frac{\ch^{\beta}_2 - \frac{\alpha^2}{2} \ch^{\beta}_0}{\ch^{\beta}_1} = \frac{\ch_2 - \beta \ch_1 + \frac{\beta^2}{2} \ch_0 - \frac{\alpha^2}{2} \ch_0}{\ch_1 - \beta \ch_0}.
\]
As in classical slope stability an object $E \in \Coh^{\beta}(\P^3)$ is called \textit{$\nu_{\alpha, \beta}$-(semi)stable} or \textit{tilt (semi)stable} with respect to $(\alpha, \beta)$ if for all short exact sequences $0 \to F \to E \to G \to 0$ in $\Coh^{\beta}(\P^3)$ the inequality $\nu_{\alpha, \beta}(F) < (\leq) \nu_{\alpha, \beta}(G)$ holds. Note that in regard to \cite{BMT14} this slope has been modified by switching $\alpha$ with $\sqrt{3} \alpha$. We prefer this point of view because it will make the walls semicircular. In concrete computations it becomes relevant to restrict the Chern characters of semistable objects. One of the main tools to perform this restriction is the following inequality for semistable objects.

\begin{thm}[{Bogomolov-Gieseker Inequality for Tilt Stability, \cite[Corollary 7.3.2]{BMT14}}]
\label{thm:bg_inequality}
Any $\nu_{\alpha, \beta}$-semistable object $E \in \Coh^{\beta}(\P^3)$ satisfies
\begin{align*}
Q^{\tilt}(E) &:= (\ch_1^{\beta}(E))^2 - 2\ch_0^{\beta}(E)\ch_2^{\beta}(E) \\
&= (\ch_1(E))^2 - 2\ch_0(E)\ch_2(E) \geq 0.
\end{align*}
\end{thm}

Let $v = \ch_{\leq 2}(E) = (v_0, v_1, v_2)$ for some object $E \in D^b(\P^3)$. A \textit{numerical wall} in tilt stability for $v$ is by definition induced by a class $(r,c,d) \in \Z^2 \times \tfrac{1}{2} \Z$ as the set of solutions $(\alpha, \beta)$ to the equation $\nu_{\alpha, \beta}(v) = \nu_{\alpha, \beta}(r,c,d)$, where we assume that this is a non trivial proper solution set. For example throughout this article, we will always choose $v = \ch_{\leq 2}(\II_C)$, where $C \subset \P^3$ is an elliptic quartic curve and study moduli spaces involving these objects.

A subset of a numerical wall is an \textit{actual wall} if the set of stable or semistable objects with class $v$ changes at it. Numerical walls in tilt stability satisfy Bertram's Nested Wall Theorem. For surfaces it was proved in \cite{MacA14}. A proof in the threefold case can be found in \cite{Sch15}.

\begin{thm}[Structure Theorem for Walls in Tilt Stability]
All numerical walls in the following statements are for fixed $v = (v_0, v_1, v_2)$.
\begin{enumerate}
  \item Numerical walls in tilt stability are of the form
  \[x\alpha^2 + x\beta^2 + y\beta + z = 0\]
  for $x = v_0c - v_1r$, $y = 2(v_2r - v_0d)$ and $z = 2(v_1d - v_2c)$. In particular, they are either semicircular walls with center on the $\beta$-axis or vertical rays.
  \item If two numerical walls given by $\nu_{\alpha, \beta}(r,c,d) = \nu_{\alpha, \beta}(v)$ and $\nu_{\alpha, \beta}(r',c',d') = \nu_{\alpha, \beta}(v)$ intersect for any $\alpha \geq 0$, then $(r,c,d)$, $(r',c',d')$ and $v$ are linearly dependent. In particular, the two walls are completely identical.
  \item The curve $\nu_{\alpha, \beta}(v) = 0$ is given by the hyperbola
  \[v_0\alpha^2 - v_0\beta^2 + 2v_1\beta - 2v_2 = 0.\]
  Moreover, this hyperbola intersects all semicircular walls at their top point.
  \item If $v_0 \neq 0$, there is exactly one vertical numerical wall given by $\beta = v_1/v_0$. If $v_0 = 0$ there is no vertical wall.
  \item If a numerical wall has a single point at which it is an actual wall, then all of it is an actual wall.
\end{enumerate}
\end{thm}

On smooth projective surfaces tilt stability is enough to get a Bridgeland stability condition (see \cite{Bri08, AB13}). On threefolds Bayer, Macr\`i and Toda proposed another tilt to obtain a suitable category to define a Bridgeland stability condition as follows. Let
\begin{align*}
\TT'_{\alpha, \beta} &:= \{E \in \Coh^{\beta}(\P^3) : \text{any quotient $E \onto G$ satisfies $\nu_{\alpha, \beta}(G) > 0$} \}, \\
\FF'_{\alpha, \beta} &:=  \{E \in \Coh^{\beta}(\P^3) : \text{any subobject $F \into E$ satisfies $\nu_{\alpha, \beta}(F) \leq 0$} \}
\end{align*}
and set $\AA^{\alpha, \beta} := \langle \FF'_{\alpha, \beta}[1], \TT'_{\alpha, \beta} \rangle $. For any $s>0$ they define
\begin{align*}
Z_{\alpha,\beta,s} &:= -\ch^{\beta}_3 + (s+\tfrac{1}{6})\alpha^2 \ch^{\beta}_1 + i (\ch^{\beta}_2 - \frac{\alpha^2}{2} \ch^{\beta}_0), \\
\lambda_{\alpha,\beta,s} &:= -\frac{\Re(Z_{\alpha,\beta,s})}{\Im(Z_{\alpha,\beta,s})}.
\end{align*}
In order to prove that this yields a Bridgeland stability condition, Bayer, Macr\`i, and Toda conjectured a generalized Bogomolov-Gieseker inequality involving third Chern characters for tilt semistable objects with $\nu_{\alpha, \beta} = 0$. In \cite{BMS16} it was shown that the conjecture is equivalent to a more general inequality that drops the hypothesis $\nu_{\alpha, \beta} = 0$. In the case of $\P^3$ the inequality was proved in \cite{MacE14}. Recall the definition of $Q^{\tilt}$ from Theorem \ref{thm:bg_inequality}.

\begin{thm}[BMT Inequality]
Any $\nu_{\alpha, \beta}$-stable object $E \in \Coh^{\beta}(\P^3)$ satisfies
\[ \alpha^2 Q^{\tilt}(E) + 4(\ch_2^{\beta}(E))^2 - 6 \ch_1^{\beta}(E) \ch_3^{\beta}(E) \geq 0. \]
\end{thm}

Similar inequalities were proved for the smooth quadric threefold \cite{Sch14} and all abelian threefolds \cite{BMS16, MP13a, MP13b}. Recently, the inequality has also been generalized to all Fano threefolds of Picard rank $1$ in \cite{Li15}. By using the definition of $\ch^{\beta}(E)$, one finds $x(E),y(E) \in \R$ such that the BMT Inequality becomes
\[\alpha^2 Q^{\tilt}(E) + \beta^2 Q^{\tilt}(E) + x(E)\beta + y(E) \geq 0.\]
This means the solution set is given by the complement of a semi-disc with center on the $\beta$-axis or a quadrant to one side of a vertical line.

Using the same proof as in the surface case in \cite[Proposition 14.1]{Bri08} leads to the following lemma. It allows to identify the moduli space of slope stable sheaves as a moduli space of tilt stable objects.

\begin{lem}
\label{lem:stable_sheaves_stable_in_tilt}
Let $v = (v_0, v_1, v_2, v_3) \in K_0(\P^3)$ such that $\beta < \mu(v)$ and $(v_0, v_1)$ is primitive. Then an object $E$ with $\ch(E) = v$ is $\nu_{\alpha, \beta}$-stable for all $\alpha \gg 0$ if and only if $E$ is a slope stable sheaf.
\end{lem}

An important question is how moduli spaces change set theoretically at walls in Bridgeland stability. In case the destabilizing subobject and quotient are both stable this has a satisfactory answer, and a proof can for example be found in \cite[Lemma 3.10]{Sch15}. Note that this does not work in the case of tilt stability due to the lack of unique Jordan-H\"older filtrations.

\begin{lem}

\label{lem:wall_crossing}
Let $\sigma = (\AA, Z) \in \Stab(\P^3)$ such that there are stable object $F,G \in \AA$ with $\mu_{\sigma}(F) = \mu_{\sigma}(G)$. Then there is an open neighborhood $U$ around $\sigma$ where non trivial extensions $0 \to F \to E \to G \to 0$ are stable for all $\sigma' \in U$ where $F \into E$ does not destabilize $E$.
\end{lem}

Another crucial issue is the construction of reasonably behaved moduli spaces of Bridgeland stable objects.  A recent result by Piyaratne and Toda is a major step towards this. It applies in particular to the case of $\P^3$, since the conjectural BMT-inequality is known.

\begin{thm}[{\cite{PT15}}]
Let $X$ be a smooth projective threefold such that the conjectural construction of Bridgeland stability from \cite{BMT14} works. Then any moduli space of semistable objects for such a Bridgeland stability condition is a universally closed algebraic stack of finite type over $\C$.
\end{thm}

If there are no strictly semistable objects, the moduli space becomes a proper algebraic space of finite type over $\C$.

Our strategy to compute concrete wall crossing follows that of \cite{Sch15}. We do numerical computations in tilt stability and then translate them into Bridgeland stability. Let $v = (v_0, v_1, v_2, v_3)$ be the Chern character of an object in $D^b(X)$. For any $\alpha > 0$, $\beta \in \R$ and $s > 0$ we denote the set of $\lambda_{\alpha, \beta, s}$-semistable objects with Chern character $v$ by $M_{\alpha, \beta, s}(v)$ and the set of $\nu_{\alpha, \beta}$-semistable objects with Chern character $v$ by $M^{\tilt}_{\alpha, \beta, s}(v)$. Analogously to our notation for twisted Chern characters, we write $v^{\beta} := (v^{\beta}_0, v^{\beta}_1, v^{\beta}_2, v^{\beta}_3) = v \cdot e^{-\beta H}$. We also write
\[
P_v := \{(\alpha, \beta) \in \R_{\geq 0} \times \R : \nu_{\alpha, \beta}(v) > 0 \}.
\]
We need the following technical statement. Under mild hypotheses, it says that on one side of the hyperbola $\{ \nu_{\alpha, \beta}(v) = 0 \}$ all the chambers and walls of tilt stability occur in Bridgeland stability. Note that $\nu_{\alpha, \beta}(v) = 0$ implies $\lambda_{\alpha, \beta, s}(v) = \infty$. This is a crucial fact in establishing the following relation between walls in tilt stability and walls in Bridgeland stability.

\begin{thm}[{\cite[Theorem 6.1]{Sch15}}]
\label{thm:wall_intersecting_hyperbola}
Let $\alpha_0 > 0$, $\beta_0 \in \R$ and $s > 0$ such that $\nu_{\alpha_0, \beta_0}(v) = 0$ and $v^{\beta_0}_1 > 0$.
\begin{enumerate}
  \item Assume there is an actual wall in Bridgeland stability for $v$ at $(\alpha_0, \beta_0)$ given by
  \[0 \to F \to E \to G \to 0.\]
  That means $\lambda_{\alpha_0, \beta_0, s}(F) = \lambda_{\alpha_0, \beta_0, s}(G)$ and $\ch(E) = -v$ for semistable $E,F,G \in \AA^{\alpha_0, \beta_0}(\P^3)$. Further assume there is a neighborhood $U$ of $(\alpha_0, \beta_0)$ such that the same sequence also defines an actual wall in $U \cap P_v$, i.e. $E,F,G$ remain semistable in $U \cap P_v \cap \{ \lambda_{\alpha, \beta, s}(F) = \lambda_{\alpha, \beta, s}(G)\}$. Then $E[-1]$, $F[-1]$, $G[-1] \in \Coh^{\beta_0}(\P^3)$ are $\nu_{\alpha_0, \beta_0}$-semistable. In particular, there is an actual wall in tilt stability at $(\alpha_0, \beta_0)$.
  \item Assume that all $\nu_{\alpha_0, \beta_0}$-semistable objects with class $v$ are stable. Then there is a neighborhood $U$ of $(\alpha_0, \beta_0)$ such that
  \[M_{\alpha, \beta, s}(v) = M^{\tilt}_{\alpha, \beta}(v)\]
  for all $(\alpha, \beta) \in U \cap P_v$. Moreover, in this case all objects in $M_{\alpha, \beta, s}(v)$ are $\lambda_{\alpha, \beta, s}$-stable.
  \item Assume there is an actual wall in tilt stability for $v$ at $(\alpha_0, \beta_0)$ given by
  \[0 \to F^n \to E \to G^m \to 0\]
  such that $F, G \in \Coh^{\beta_0}(\P^3)$ are $\nu_{\alpha_0, \beta_0}$-stable objects, $\ch(E) = v$ and $\nu_{\alpha_0, \beta_0}(F) = \nu_{\alpha_0, \beta_0}(G)$. Assume further that the set
  \[P_v \cap P_{\ch(F)} \cap P_{\ch(G)} \cap \{ \lambda_{\alpha, \beta, s}(F) = \lambda_{\alpha, \beta, s}(G)\}\]
  is non-empty. Then there is a neighborhood $U$ of $(\alpha_0, \beta_0)$ such that $F,G$ are $\lambda_{\alpha, \beta, s}$-stable for all $(\alpha, \beta) \in U \cap P_v \cap \{ \lambda_{\alpha, \beta, s}(F) = \lambda_{\alpha, \beta, s}(G)\}$. In particular, there is an actual wall in Bridgeland stability in $U \cap P_v$ defined by the same sequence.
\end{enumerate}
\end{thm}

This Theorem will be used as follows in the the remainder of the article. Assume that we have determined all exact sequences that give walls in tilt stability for objects with a fixed Chern character $v$. By part (1) of the Theorem, we know that on one side of the hyperbola $\nu_{\alpha, \beta}(v) = 0$ the only walls in Bridgeland stability have to be defined by an exact sequence giving a wall in tilt stability. We will then use part (3) to show that every such sequence does indeed define a wall in Bridgeland stability. At this point we know all exact sequences defining walls on a path close to one side of the hyperbola $\nu_{\alpha, \beta}(v) = 0$. Finally, we have to use part (2) to show that all the moduli spaces of tilt stable objects actually occur in Bridgeland stability on this path.

By doing this, we can translate simple computations in tilt stability into the more complicated framework of Bridgeland stability. Sometimes there are exact sequences giving identical numerical walls in tilt stability, but different numerical walls in Bridgeland stability. Therefore, this translation allows us to observe additional chambers that are hidden in tilt stability.

\section{Tilt Stability for Elliptic Quartics}
Let $C$ be the complete intersection of two quadrics in $\P^3$, i.e. an elliptic quartic curve. We will compute all walls in tilt stability for $\beta < 0$ with respect to $v = \ch(\II_C)$. There is a locally free resolution $0 \to \OO(-4) \to \OO(-2)^{\oplus 2} \to \II_C \to 0$. This leads to
\[\ch^{\beta}(\II_C) = \left(1, -\beta, \frac{\beta^2}{2} - 4, -\frac{\beta^3}{6} + 4\beta + 8\right).\]
We denote the set of tilt semistable objects with respect to $(\alpha, \beta)$ and class $v$ by $M^{\tilt}_{\alpha, \beta}(v)$.

\begin{thm}
\label{thm:tiltStabilityEllipticQuartics}
There are three walls for $M^{\tilt}_{\alpha, \beta}(1,0,-4,8)$ for $\alpha > 0$ and $\beta < 0$. Moreover, the following table lists pairs of tilt semistable objects whose extensions completely describe all strictly semistable objects at each of the corresponding walls. Let $L$ be a line in $\P^3$, $V$ a plane in $\P^3$, $Z \subset \P^3$ a length two zero dimensional subscheme, $Z' \subset
V$ a length two zero dimensional subscheme and $P \in \P^3$, $Q \in V$ be points.
\begin{center}
  \begin{tabular}{ r | l }
    & \\
    $\alpha^2 + (\beta + 3)^2 = 1$ & $\OO(-2)^{\oplus
    2}$, $\OO(-4)[1]$ \\
    & \\
    \hline
    & \\
    $\alpha^2 + \left(\beta + \frac{7}{2} \right)^2 = \frac{17}{4}$ &
    $\II_L(-1)$, $\OO_V(-3)$ \\ 
    & \\
    \hline
    & \\
    \ & $\II_Z(-1)$, $\OO_V(-4)$ \\
    & \\
    $\alpha^2 + \left(\beta + \frac{9}{2}\right)^2 =
    \left(\frac{7}{2}\right)^2$ & $\II_P(-1)$, $\II_{Q/V}(-4)$ \\
    & \\
    \ & $\OO(-1)$, $\II_{Z'/V}(-4)$ \\
  \end{tabular}
\end{center}
The hyperbola $\nu_{\alpha, \beta}(1,0,-4) = 0$ is given by the equation
\[\beta^2 - \alpha^2 = 8.\]
Moreover, there are no semistable objects for $(\alpha, \beta)$ inside the smallest semicircle.
\end{thm}

It is interesting to note that all relevant objects in this Theorem are sheaves and no actual 2-term complexes. The key difference to the classical picture, as we will see later, is that some sheaves of positive rank with torsion will turn out to be stable and replace ideal sheaves of heavily singular curves in some chambers.

The fact that the smallest wall is given by the equation $\alpha^2 + (\beta + 3)^2 = 1$ was already proved in \cite[Theorem 5.1]{Sch15} in more generality. Moreover, it was shown there that all semistable objects $E$ at the wall are given by extensions of the form $0 \to \OO(-2)^{\oplus 2} \to E \to \OO(4)[1] \to 0$ and that there are no tilt semistable objects inside this semicircle.

In order to prove the remainder of Theorem \ref{thm:tiltStabilityEllipticQuartics} we need to put numerical restrictions on potentially destabilizing objects. This can be done by the following two lemmas.

\begin{lem}[{\cite[Lemma 5.4]{Sch14}}]
\label{lem:idealSheafRestriction}
Let $E \in \Coh^{\beta}(\P^3)$ be tilt semistable with respect to some $\beta \in \Z$ and $\alpha \in \R_{> 0}$.
\begin{enumerate}
  \item If $\ch^{\beta}(E) = (1,1,d,e)$ then $d - 1/2 \in \Z_{\leq 0}$. In the case $d = -1/2$, we get $E \cong \II_L(\beta + 1)$ where $L$ is a line plus $1/6-e$ (possibly embedded) points in $\P^3$. If $d = 1/2$, then $E \cong \II_Z(\beta + 1)$ for a zero dimensional subscheme $Z \subset \P^3$ of length $1/6 - e$.
  \item If $\ch^{\beta}(E) = (0,1,d,e)$, then $d - 1/2 \in \Z$ and $E \cong I_{Z/V}(\beta + d + 1/2)$ where $Z$ is a dimension zero subscheme of length $1/24 + d^2/2 - e$.
\end{enumerate}
\end{lem}

The next lemma determines the Chern characters of possibly destabilizing objects for $\beta = -2$.
\begin{lem}
\label{lem:chAtMinusTwo}
If an exact sequence $0 \to F \to E \to G \to 0$ in $\Coh^{-2}(\P^3)$ defines a wall for $\beta = -2$ with $\ch_{\leq 2}(E) = (1, 0, -4)$ then
\[\ch^{-2}_{\leq 2}(F), \ch^{-2}_{\leq 2}(G) \in \left\{\left(1, 1, -\frac{1}{2}\right), \left(0, 1, -\frac{3}{2}\right), \left(1, 1, \frac{1}{2}\right), \left(0,1, -\frac{5}{2}\right) \right\}.\]
\begin{proof}
The four possible Chern characters group into two cases that add up to $\ch^{-2}_{\leq 2}(E) = (1, 2, -2)$.

Let $\ch^{-2}_{\leq 2}(F) = (r,c,d)$. By definition of $\Coh^{-2}(\P^3)$, we have $0 \leq c \leq 2$. If $c=0$, then $\nu_{\alpha, -2}(F) = \infty$ and this is in fact no wall for any $\alpha > 0$. If $c=2$, then the same argument for the quotient $G$ shows there is no wall. Therefore, $c=1$ must hold. We can compute
\begin{align*}
\nu_{\alpha, -2}(E) = -1 - \frac{\alpha^2}{4}, \ \nu_{\alpha, -2}(F) = d - \frac{r \alpha^2}{2}.
\end{align*}
The wall is defined by $\nu_{\alpha, -2}(E) = \nu_{\alpha, -2}(F)$. This leads to
\begin{align}
\label{eq:alphaPositive}
\alpha^2 = \frac{4d+4}{2r-1} > 0.
\end{align}
The next step is to rule out the cases $r \geq 2$ and $r \leq -1$. If $r \geq 2$, then $\ch_0(G) \leq -1$. By exchanging the roles of $F$ and $G$ in the following argument, it is enough to deal with the situation $r \leq -1$. In that case we use (\ref{eq:alphaPositive}) and the Bogomolov Gieseker inequality to get the contradiction $2rd \leq 1$, $d < -1$ and $r \leq -1$.

Therefore, we know $r=0$ or $r=1$. By again interchanging the roles of $F$ and $G$ if necessary we only have to handle the case $r=1$. Equation (\ref{eq:alphaPositive}) implies $d > - 1$. By Lemma
\ref{lem:idealSheafRestriction} we get $d - 1/2 \in \Z_{\leq 0}$. Therefore, we are left with the cases claimed.
\end{proof}
\end{lem}

\begin{proof}[Proof of Theorem \ref{thm:tiltStabilityEllipticQuartics}]
By assumption we are only dealing with walls that intersect the branch of the hyperbola with $\beta < 0$. As explained before, we already know the smallest wall. This semicircle intersects the $\beta$-axis at $\beta = -4$ and $\beta = -2$. Therefore, all other walls intersecting this branch of the hyperbola also have to intersect the ray $\beta = -2$. By Lemma \ref{lem:chAtMinusTwo} there are at most two walls intersecting the line $\beta = -2$. They correspond to the two solutions claimed to exist.

Let $0 \to F \to E \to G \to 0$ define a wall in $\Coh^{-2}(\P^3)$ with $\ch(E) = (1,0,-4,8)$. One can compute $\ch^{-2}(E) = (1, 2, -2, \tfrac{4}{3})$. A direct computation shows that the middle wall is given by $\ch^{-2}(F) = (1, 1, -1/2, e)$ and $\ch^{-2}(G) = (0, 1, -3/2, 4/3 - e)$. By Lemma \ref{lem:idealSheafRestriction} we get $F \cong \II_L(-1)$ where $L$ is a line plus $1/6-e$ (possibly embedded) points in $\P^3$. In particular, the inequality $e \leq 1/6$ holds. The same lemma also implies that $G \cong I_{Z/V}(-3)$ where $Z$ is a dimension zero subscheme of length $e-1/6$. Overall this shows $e=1/6$. Therefore, $L$ is a just a line and $E \cong \OO_V(-3)$.

The outermost wall is given by $\ch^{-2}(F) = (1, 1, 1/2, e)$ and $\ch^{-2}(G) = (0, 1, -5/2, 4/3 - e)$. We use again Lemma \ref{lem:idealSheafRestriction} to get $F \cong \II_Z(-1)$ for a zero
dimensional subscheme $Z \subset \P^3$ of length $1/6 - e$. Therefore, we have $e-1/6 \in \Z_{\geq 0}$. The lemma also shows $G \cong I_{Z/V}(-4)$ where $Z$ is a dimension zero subscheme of length $e + 11/6$. Overall, we get $e \in \{-11/6, -5/6, 1/6\}$. That corresponds exactly to the three cases in the Theorem.
\end{proof}

\section{Curves on the intersection of the two components}

Let $\Hilb^{4t}_1 \subset \Hilb^{4t}(\P^3)$ be the closure of the locus of smooth elliptic quartic curves. By $\Hilb^{4t}_2 \subset \Hilb^{4t}(\P^3)$ we denote the closure of the locus of plane quartics curves plus two disjoint points. A straightforward dimension count shows $\dim \Hilb^{4t}_1 = 16$ and $\dim \Hilb^{4t}_2 = 23$. In this section, we will prove some preliminary results about the intersection of the two components. We will do this following the approach of Piene and Schlessinger in \cite{PS85}, which requires a careful analysis of the equations of the curves along this intersection.

\begin{prop}
\label{prop:ideal_sequence_three}
Let $I_C$ be the ideal of a subscheme $C \subset \P^3$ of dimension $1$, which fits into an exact sequence of the form $0 \to \II_{Z'}(-1) \to I_C \to \OO_V(-4) \to 0$, where $V$ is a plane in $\P^3$ and $Z' \subset V$ is a zero dimensional subscheme of length two.
\begin{enumerate}
\item The ideal $I_C$ is projectively equivalent to one of the ideals
\begin{align*}
(x^2, xy, xzw, f_4(x,y,z,w)), \\
(x^2, xy, xz^2, g_4(x,y,z,w)),
\end{align*}
where $f_4 \in (x, y, zw)$, respectively $g_4 \in (x, y, z^2)$ is of degree $4$.
\item The ideal 
\[ (x^2, xy, xz^2, y^4) \]
lies in the closure of the orbit of $\II_C$ under the action of $\PGL(4)$ for any $\II_C$ as above.
\end{enumerate}
\begin{proof}
Up to the action of $\PGL(4)$ we can assume that either $I_{Z'} = (x, y, zw)$ or $I_{Z'} = (x, y, z^2)$ and $I_V = (x)$. The exact sequence $0 \to \II_{Z'}(-1) \to I_C \to \OO_V(-4) \to 0$ implies that either $l(x, y, z, w) \cdot (x, y, zw) \subset I_C$ or $l(x, y, z, w) \cdot (x, y, z^2) \subset I_C$ for a linear polynomial $l(x, y, z, w) \in \C[x,y,z,w]$. Since the quotient is supported on $V$, we must have $l = x$. Therefore, either $(x^2, xy, xzw) \subset I_C$ or $(x^2, xy, xz^2) \subset I_C$. Since the quotient is $\OO_V(-4)$, there has to be another degree $4$ generator $f_4(x,y,z,w)$ with $x f_4(x,y,z,w) \in (x^2, xy, xzw)$, respectively $g_4(x,y,z,w)$ with $x g_4(x,y,z,w) \in (x^2, xy, xz^2)$. That proves (1).

By (1), we can assume that either $I_C = (x^2, xy, xzw, f_4(x,y,z,w))$ for $f_4 \in (x, y, zw)$ or $I_C = (x^2, xy, xz^2, g_4(x,y,z,w))$ for $g_4 \in (x, y, z^2)$. We can take the limit $t \to 0$ for the action of the element $g_t \in \PGL(4)$ that fixes $x$, $y$, $z$ and maps $w \mapsto (1-t)z + tw$. Thus, we can assume that $I_C = (x^2, xy, xz^2, g_4(y,z))$ where $g_4 \in \C[y,z]$. Pick $\lambda \in \C \backslash \{ 0 \}$ such that $g_4(\lambda, 1) \neq 0$. We analyze the action of $g_t \in \PGL(4)$ that fixes $x$, $w$, maps $y \mapsto \lambda y$ and maps $z \mapsto (1-t)y + tz$. We get
\begin{align*}
g_t \cdot (x^2, xy, xz^2, g_4(y,z)) &= (x^2, \lambda xy, (1-t)^2 xy^2 +2(1-t)txyz + t^2 xz^2, g_4(\lambda y, (1-t)y + tz)) \\
&= (x^2, xy, xz^2, g_4(\lambda y, (1-t)y + tz)).
\end{align*}
Since $g_4(\lambda, 1) \neq 0$, we have $g_4(\lambda y, y) \neq 0$ and we can finish the proof of (2) by taking the limit $t \to 0$.
\end{proof}
\end{prop}

Next we want to analyze the singularities of the point on the Hilbert scheme corresponding to $(x^2, xy, xz^2, y^4)$. We will use \cite{M2} and the techniques developed in \cite{PS85}.

\begin{prop}
\label{prop:piene-schlessinger}
If $I_C = (x^2, xy, xz^2, y^4)$, then $I_C$ lies on the intersection of two irreducible components of $\Hilb(\P^3)$ and is a smooth point on each of them. Moreover, the intersection is locally of dimension $15$ and transversal.
\begin{proof}
Let $p_C \in \Hilb(\P^3)$ be the point parametrizing $C$. Next, we use the comparison theorem \cite[p. 764]{PS85} which claims the Hilbert scheme $\Hilb(\P^3)$ and the universal deformation space which parametrizes all homogeneous ideals with Hilbert function equal to that of $I_C$
are isomorphic in an \'etale neighborhood of the point $p_C$ if 
\[
\left( \frac{\mathbb{C}[x,y,z,w]}{I_C} \right)_d \cong H^0(C, \mathcal{O}_C(d))
\]  
for $d=\deg(f_i)$ where $f_i$ are generators of $I_C$. For our particular ideal, this equality can for example directly be checked with help of Macaulay2 or by hand. The comparison theorem allows us to find local equations of the Hilbert scheme near $p_C$ by using the same strategy than the proof of \cite[Lemma 6]{PS85}. In fact, this procedure has been implemented in the Macaulay2 Package ``VersalDeformations" (see \cite{Ilt12}). In particular, the routine localHilbertScheme generates an ideal of the form (see Appendix A)
\[
\left( -t_5t_{24}, -t_6t_{24} ,-t_7t_{24}, -t_8t_{24},t_{15}t_{24} ,t_{16}t_{24},t_{17}t_{24}-2t_{22}t_{24},t_{18}t_{24}-2t_{23}t_{24} \right) \in \mathbb{C}[t_1, \ldots, t_{24}.]
\]
Then, \'etale locally at $p_C$, the Hilbert scheme is the transversal intersection of the hyperplane $(t_{24}=0)$ and a $16$-dimensional linear subspace.
\end{proof}
\end{prop}

It is not hard to see that the two components $(x^2, xy, xz^2, y^4)$ is lying on are $\Hilb^{4t}_1$ and $\Hilb^{4t}_2$ by giving explicit degenerations. However, it is also a direct consequence of the results in the next section.

\section{Bridgeland stability}
\label{sec:bridgeland}

The goal of this section is to translate the computations in tilt stability to actual wall crossings in Bridgeland stability. We will analyze the singular loci of the occurring moduli spaces and use this to reprove the global description of the main component of the Hilbert scheme as in \cite{VA92}.

As a consequence of Theorem \ref{thm:wall_intersecting_hyperbola} and Theorem \ref{thm:tiltStabilityEllipticQuartics}, we obtain the following corollary. In this application of Theorem \ref{thm:wall_intersecting_hyperbola} all exact sequences giving walls in tilt stability to the left hand side of the unique vertical wall are of the form in (3). Therefore, we do not have more sequences giving walls in tilt stability than in Bridgeland stability to the left hand side of the left branch of the hyperbola.

\begin{cor}
\label{cor:allwalls}
There is a path $\gamma: [0,1] \to \R_{>0} \times \R \subset \Stab(\P^3)$ that crosses the following walls for $v = (1, 0, -4, 8)$ in the given order. The walls are defined by the two given objects having the same slope. Moreover, all strictly semistable objects at each of the walls are extensions of those two objects. Let $L$ be a line in $\P^3$, $V$ a plane in $\P^3$, $Z \subset \P^3$ a length two zero dimensional subscheme, $Z' \subset V$ a length two zero dimensional subscheme and $P \in \P^3$, $Q \in V$ be points.
\begin{enumerate}
  \item $\OO(-2)^{\oplus 2}$, $\OO(-4)[1]$
  \item $\II_L(-1)$, $\OO_V(-3)$
  \item $\II_Z(-1)$, $\OO_V(-4)$
  \item $\II_P(-1)$, $\II_{Q/V}(-4)$
  \item $\OO(-1)$, $\II_{Z'/V}(-4)$
\end{enumerate}
\end{cor}

We denote the moduli space of Bridgeland stable objects with Chern character $(1,0,-4,8)$ in the chambers from inside the smallest wall to outside the largest wall by $M_0, \ldots, M_5$. The goal of this section is to give some description of these spaces. By Theorem \ref{thm:tiltStabilityEllipticQuartics} we have $M_0 = \emptyset$. After the largest wall we must have $M_5 = \Hilb^{4t}(\P^3)$. More precisely, it is the moduli of ideal sheaves which is the same as the Hilbert scheme due to \cite[p. 1265]{MNOP06}. See Figure \ref{fig:bridgeland_walls_elliptic} for a visualization of the walls.

\begin{prop}
The first moduli space $M_1$ is isomorphic to the Grassmannian $\G(1,9)$.
\begin{proof}
All extensions in $\Ext^1(\OO(-4)[1], \OO(-2)^{\oplus 2})$ are cokernels of morphisms $\OO(-4) \to \OO(-2)^{\oplus 2}$. The stability condition ensures that the two quadrics defining it are not collinear. Therefore, these extensions parametrize pencils of quadrics and the moduli space is the Grassmannian $\G(1,9)$.
\end{proof}
\end{prop}

The tangent space of a moduli space of Bridgeland stable objects at any stable complex $E$ is given by $\Ext^1(E,E)$ (see \cite{Ina02} and \cite{Lie06} for the deformation theory of moduli spaces of complexes). Obtaining these groups requires a substantial amount of diagram chasing and computations. In order to minimize the distress on the reader and the authors, we will prove the following lemma with heavy usage of \cite{M2}.

\begin{lem}
\label{lem:ext-computations}
Let notation be as in Theorem \ref{cor:allwalls}. The equalities
\begin{align*}
\Ext^1(\II_L(-1), \OO_V(-3)) = \C &, \ \Ext^1(\OO_V(-3), \II_L(-1)) = \C^9, \\
\Ext^1(\II_L(-1), \II_L(-1)) = \C^4 &, \ \Ext^1(\OO_V(-3), \OO_V(-3)) = \C^3, \\
\Ext^1(\II_Z(-1), \OO_V(-4)) = \begin{cases} \C &, \ Z \subset V \\ 0 &, \text{ otherwise} \end{cases} &, \ \Ext^1(\OO_V(-4), \II_Z(-1)) = \C^{15}, \\
\Ext^1(\II_Z(-1), \II_Z(-1)) = \C^6 &, \ \Ext^1(\OO_V(-4), \OO_V(-4)) = \C^3, \\ 
\Ext^1(\II_P(-1), \II_{Q/V}(-4)) = \begin{cases} \C^3 &, \ P = Q \\ \C &, \ P \neq Q \end{cases} &, \ \Ext^1(\II_{Q/V}(-4), \II_P(-1)) = \begin{cases} \C^{17} &, \ P = Q \\ \C^{15} &, \ P \neq Q, \end{cases} \\
\Ext^1(\II_P(-1), \II_P(-1)) = \C^3 &, \ \Ext^1(\II_{Q/V}(-4), \II_{Q/V}(-4)) = \C^5, \\ 
\Ext^1(\OO(-1), \II_{Z'/V}(-4)) = \C^2 &, \ \Ext^1(\II_{Z'/V}(-4), \OO(-1)) = \C^{15}, \\
\Ext^1(\OO(-1), \OO(-1)) = 0 &, \ \Ext^1(\II_{Z'/V}(-4), \II_{Z'/V}(-4)) = \C^7
\end{align*}
hold.
If $Z \subset V$ is a double point supported at $P$, then
\begin{align*}
\Ext^1(\II_Z(-1), \II_{P/V}(-4)) = \C^3 &, \ \Ext^1(\OO_V(-4)), \II_{P/V}(-4)) = \C^2, \\
\Ext^1(\II_Z(-1), \II_P(-1)) = \C^3 &, \ \Ext^1(\OO_V(-4)), \II_P(-1)) = \C^{15}.
\end{align*}
\begin{proof}
Under the action of $\PGL(4)$ there are two orbits of pairs of a line and a plane $(L,V)$. Either we have $L \subset V$ or not. By choosing representatives defined over $\Q$, we can use \cite{M2} to compute $\Ext^1(\II_L(-1), \OO_V(-3)) = \C$, $\Ext^1(\OO_V(-3), \II_L(-1)) = \C^9$, $\Ext^1(\OO_V(-3), \OO_V(-3) = \C^3$ and $\Ext^1(\II_L(-1), \II_L(-1) = \C^4$. All other equalities follow in the same way. The Macaulay2 code can be found in Appendix A.
\end{proof}
\end{lem}

Since the dimension of tangent spaces is bounded from below by the dimension of the space, the following lemma can sometimes simplify computations.

\begin{lem}
\label{lem:ext_estimate}
Let $0 \to F^n \to E \to G^m \to 0$ be an exact sequence at a wall in Bridgeland stability where $F$ and $G$ are distinct stable objects of the same Bridgeland slope and $E$ is semistable to one side of the wall. Then the following inequality holds,
\[\ext^1(E,E) \leq n^2 \ext^1(F,F) + m^2 \ext^1(G,G) + nm \ext^1(F,G) + nm\ext^1(G,F) - n^2. \]
\begin{proof}
Stability to one side of the wall implies $\Hom(E,F) = 0$. Since $F$ is stable, we also know $\Hom(F,F) = \C$. By the long exact sequence coming from applying $\Hom(\cdot, F)$ to the above exact sequence, we get $\ext^1(E, F) \leq m \ext^1(G, F) + n \ext^1(F, F) - n$. Moreover, we can use $\Hom(\cdot, G)$ to get $\ext^1(E, G) \leq m \ext^1(G, G) + n \ext^1(F, G)$. These two inequalities together with applying $\Hom(E, \cdot)$ lead to the claim.
\end{proof}
\end{lem}

We also have to handle the issue of potentially new components after crossing a wall. The following result will solve this issue in some cases.

\begin{lem}
\label{lem:new_components}
Let $M$ and $N$ be two moduli spaces of Bridgeland semistable objects separated by a single wall. Assume that $A \subset M$ and $B \subset N$ are the loci destabilized at the wall. If $A$ intersects an irreducible component $H$ of $M$ non trivially and $H$ is not contained in $A$, then $B$ must intersect the closure of $H \backslash A$ inside $N$.
\begin{proof}
This follows from the fact that moduli spaces of Bridgeland semistable objects are universally closed. If $B$ would not intersect the closure of $H \backslash A$ inside $N$, then this would correspond to a component in $N$ that is not universally closed.
\end{proof}
\end{lem}

In order to identify the global structure of some of the moduli spaces as blow ups we need the following classical result by Moishezon. Recall that the analytification of a smooth proper algebraic spaces of finite type over $\C$ of dimension $n$ is a complex manifold with $n$ independent meromorphic functions.

\begin{thm}[\cite{Moi67}]
\label{thm:blow_up}
Any birational morphism $f: X \to Y$ between smooth proper algebraic spaces of finite type over $\C$ such that the contracted locus $E$ is irreducible and the image $f(E)$ is smooth is the blow up of $Y$ in $f(E)$.
\end{thm}

\begin{prop}
\label{prop:second_moduli}
The second moduli space $M_2$ is the blow up of $\G(1,9)$ along the smooth locus $\G(1,3) \times (\P^3)^{\vee}$ parametrizing pairs $(\II_L(-1), \OO_V(-3))$. The center of the blow up parametrizes pencils whose base locus is not of dimension one. A generic point of the exceptional divisor parametrizes the union of a line and a plane cubic which intersect themselves at a point. The only non-ideal sheaves in the moduli space come from the case when the line is contained in the plane.
\begin{proof}
We know that $M_1$ is smooth. The wall separating $M_1$ and $M_2$ has strictly semistable objects given by extensions between $\II_L(-1)$ and $\OO_V(-3)$. By Lemma \ref{lem:ext-computations} we have $\Ext^1(\II_L(-1), \OO_V(-3)) = \C$, $\Ext^1(\OO_V(-3), \II_L(-1)) = \C^9$, $\Ext^1(\OO_V(-3), \OO_V(-3)) = \C^3$, and $\Ext^1(\II_L(-1), \II_L(-1)) = \C^4$.

This means the locus of semistable objects occurring as extensions in $\Ext^1(\II_L(-1), \OO_V(-3))$ for any $L$ and $V$ is isomorphic to $\G(1,3) \times (\P^3)^{\vee}$, i.e. is smooth and irreducible. By Lemma \ref{lem:wall_crossing} this is the locus destabilized at the wall in $\G(1,9)$. By Lemma \ref{lem:ext_estimate} any extension $E$ in $\Ext^1(\OO_V(-3), \II_L(-1))$ satisfies $\ext^1(E,E) \leq 16$. Lemma \ref{lem:new_components} shows that $M_2$ has to be connected, i.e. is smooth and irreducible. The locus of semistable objects that can be written as extensions in $\Ext^1(\OO_V(-3), \II_L(-1))$ for any $L$ and $V$ is irreducible of dimension $15$, i.e. is a divisor in $M_2$. An immediate application of Theorem \ref{thm:blow_up} implies the fact that $M_2$ is the blow up of $\G(1,9)$ in the smooth locus $\G(1,3) \times (\P^3)^{\vee}$.

The description of the exceptional divisor is immediate from the fact that curves $C$ with ideal sheaves fitting into an exact sequence
\[
0 \to \II_L(-1) \to \II_C \to \OO_V(-3) \to 0
\]
have to be unions of lines with a plane cubic intersecting in one point. If $L \subset V$, then no such extension can be an ideal sheaf, since the line would intersect the cubic in three points giving the wrong genus.
\end{proof}
\end{prop}

The next moduli space will acquire a second component.

\begin{prop}
\label{prop:third_moduli}
The third moduli space $M_3$ has two irreducible components $M_3^1$ and $M_3^2$. The first component $M_3^1$ is the blow up of $M_2$ in the smooth incidence variety parametrizing length two subschemes in a plane in $\P^3$. The second component $M_3^2$ is a $\P^{14}$-bundle over $\Hilb^2(\P^3) \times (\P^3)^{\vee}$ parametrizing pairs $(\II_Z(-1)$, $\OO_V(-4))$. It generically parametrizes unions of plane quartics with two generic points in $\P^3$. The two components intersect transversally along the exceptional locus of the blow up. The only non-ideal sheaves occur in the case where at least one of the two points is not scheme-theoretically contained in the plane.
\begin{proof}
By Lemma \ref{lem:ext-computations} we have
\begin{align*}
\Ext^1(\II_Z(-1), \OO_V(-4)) &= \begin{cases} \C &, \ Z \subset V \\ 0 &, \text{ otherwise,} \end{cases} \\
\Ext^1(\OO_V(-4), \II_Z(-1)) &= \C^{15}.
\end{align*}
This means the locus destabilized in $M_2$ is of dimension $7$, and the new locus appearing in $M_3$ is of dimension $23$. Since $M_2$ is of dimension $16$, the locus appearing in $M_3$ must be a new component $M_3^2$. The closure of what is left of $M_2$ is denoted by $M_3^1$.  If $M_3^2$ is reduced, it is a $\P^{14}$-bundle over $\Hilb^2(\P^3) \times (\P^3)^{\vee}$ parametrizing pairs $(\II_Z(-1)$, $\OO_V(-4))$. We will more strongly show that it is smooth.

Assume $Z$ is not scheme theoretically contained in $V$. Then Lemma \ref{lem:ext_estimate} implies that any non-trivial extension $E$ in $\Ext^1(\OO_V(-4), \II_Z(-1))$ satisfies $\ext^1(E,E) \leq 23$. Therefore, it is a smooth point and can in particular not lie on $M_3^1$. Let $E$ be an extensions of the form $0 \to \II_Z(-1) \to E \to \OO_V(-4) \to 0$, where $Z \subset V$. Any point on the intersection must satisfy $\ext^1(E,E) \geq 24$. Assume $E$ is not an ideal sheaf. If $E$ fits into an exact sequence $0 \to \II_{Z/V}(-4) \to E \to \OO(-1) \to 0$ or $0 \to \II_{Q/V}(-4) \to E \to \II_P(-1) \to 0$ for $P \neq Q$, then a direct application of Lemma \ref{lem:ext_estimate} to these sequences shows $\ext^1(E,E) \leq 23$, a contradiction. Therefore, $E$ must fit into an exact sequence $0 \to \II_{P/V}(-4) \to E \to \II_P(-1) \to 0$. Then we have the following commutative diagram with short exact rows and columns.

\centerline{
\xymatrix{
0 \ar@{^{(}->}[r] \ar@{^{(}->}[d] & \II_{P/V}(-4) \ar@{->>}[r] \ar@{^{(}->}[d] & \II_{P/V}(-4) \ar@{^{(}->}[d] \\
\II_Z(-1) \ar@{^{(}->}[r] \ar@{->>}[d] & E \ar@{->>}[r] \ar@{->>}[d] & \OO_V(-4) \ar@{->>}[d] \\
\II_Z(-1) \ar@{^{(}->}[r] & \II_P(-1) \ar@{->>}[r] & \OO_P \\
}}

Therefore, $Z$ has to be a double point supported at $P$. By Lemma \ref{lem:ext-computations} we have
\begin{align*}
\Ext^1(\II_Z(-1), \II_{P/V}(-4)) = \C^3 &, \ \Ext^1(\OO_V(-4)), \II_{P/V}(-4)) = \C^2, \\
\Ext^1(\II_Z(-1), \II_P(-1)) = \C^3 &, \ \Ext^1(\OO_V(-4)), \II_P(-1)) = \C^{15}.
\end{align*}
Next, we apply $\Hom(\cdot, \II_{P/V}(-4))$ to  $0 \to \II_Z(-1) \to E \to \OO_V(-4) \to 0$ to get $\ext^1(E, \II_{P/V}(-4)) \leq 5$. By applying $\Hom(\cdot, \II_P(-1))$ to the same sequence we get $\ext^1(E, \II_P(-1)) \leq 18$. Finally, we can apply $\Hom(E, \cdot)$ to $0 \to \II_{P/V} \to E \to \II_P(-1) \to 0$ to get $\ext^1(E,E) \leq 23$.

Therefore, the intersection of $M_3^1$ and $M_3^2$ parametrizes some of the ideals fitting into an exact sequence $0 \to \II_Z(-1) \to I_C \to \OO_V(-4) \to 0$, where $Z \subset V$. The intersection must have a closed orbit. By Proposition \ref{prop:ideal_sequence_three}, there is precisely one such closed orbit. If the intersection was disconnected, it would have at least two closed orbits. If it is reducible, then the closed orbit must lie on the intersection of all irreducible components. By Proposition \ref{prop:piene-schlessinger} the intersection along the closed orbit is transversal of dimension $15$, and its points are smooth on both components. That would be impossible if the intersection is not irreducible at the closed orbit. The singular locus on either component is closed and must therefore contain a closed orbit. Thus, the whole intersection must consist of points that are smooth on each of the two components individually. The induced map $M_3^1 \to M_2$ contracts the intersection, which is an irreducible divisor, onto a locus isomorphic to the smooth incidence variety parametrizing length two subschemes in a plane in $\P^3$. Theorem \ref{thm:blow_up} implies the description of $M_3^1$.

The description of the curves parametrized by $M_3^2$ is again a consequence of the exact sequence that the ideal sheaves fit into.
\end{proof}
\end{prop}

In order to reprove the description of the main component of the Hilbert scheme from \cite{VA92}, we have to make sure that none of the remaining walls modify the first component.

\begin{prop}
The fourth moduli space $M_4$ has two irreducible components $M_4^1$ and $M_4^2$. The first component is equal to $M_3^1$. The second component is birational to $M_3^2$.
\begin{proof}
Lemma \ref{lem:ext-computations} says
\begin{align*}
\Ext^1(\II_P(-1), \II_{Q/V}(-4)) &= \begin{cases} \C^3 &, \ P = Q \\ \C &, \ P \neq Q \end{cases}, \\
\Ext^1(\II_{Q/V}(-4), \II_P(-1)) &= \begin{cases} \C^{17} &, \ P = Q \\ \C^{15} &, \ P \neq Q \end{cases}.
\end{align*}
Moreover, the moduli space of pairs $(\II_P(-1), \II_{Q/V}(-4))$ is irreducible of dimension $8$, while the sublocus where $P = Q$ is of dimension $5$. Therefore, the closure of the locus of extensions in $\Ext^1(\II_{Q/V}(-4), \II_P(-1))$ for $P \neq Q$ is irreducible of dimension $22$. The locus of extensions in $\Ext^1(\II_{P/V}(-4), \II_P(-1))$ for $P \in V$ is irreducible of dimension $21$. Let $M_4^1$ be the closure of what is left from $M_3^1$ in $M_4$ and $M_4^2$ be the closure of what is left from $M_3^2$.

If $P \neq Q$, then Lemma \ref{lem:ext_estimate} implies smoothness. In particular, we can use Lemma \ref{lem:new_components} to show that all points in $\Ext^1(\II_{Q/V}(-4), \II_P(-1))$ for $P \neq Q$ are in $M_4^2$ and no other component. Assume we have a general non trivial extension $0 \to \II_P(-1) \to E \to \II_{P/V}(-4) \to 0$. Then $E = I_C$ is an ideal sheaf of a plane quartic curve plus a double point in the plane. We can assume that the double point is not an embedded point due to the fact that $E$ is general. Clearly, $I_C$ is the flat limit of elements in $\Ext^1(\II_{Q/V}(-4), \II_P(-1))$ by choosing $P \notin V$ and regarding the limit $P \to Q$. Therefore, $E$ is a part of $M_4^2$.

We showed $M_4 = M_4^1 \cup M_4^1$ and that $M_4^2$ is birational to $M_3^2$. We are left to show $M_4^1 = M_4^2$. If not, there is an object $E$ with a non trivial exact sequence $0 \to \II_P(-1) \to E \to \II_{P/V}(-4) \to 0$ in $M_4^1$. By Lemma \ref{lem:new_components} this implies that there is also an object $E'$ with non trivial exact sequence $0 \to \II_{P/V}(-4) \to E' \to \II_P(-1) \to 0$ lying on $M_3^1$. But we already established that all those extensions are smooth points on $M_3^2$ in the previous proof.
\end{proof}
\end{prop}

We can now prove the following theorem. 

\begin{thm}
\label{thm:last_wall}
The Hilbert scheme $\Hilb^{4t}(\P^3)$ has two components $\Hilb^{4t}_1$ and $\Hilb^{4t}_2$. The main component $\Hilb^{4t}_1$ contains an open subset of elliptic quartic curves and is a smooth double blow up of the Grassmannian $\G(1,9)$. The second component is of dimension $23$. Moreover, the two components intersect transversally in a locus of dimension $15$.
\begin{proof}
By Lemma \ref{lem:ext-computations} we have
\begin{align*}
\Ext^1(\OO(-1), \II_{Z'/V}(-4)) = \C^2, \\
\Ext^1(\II_{Z'/V}(-4), \OO(-1)) = \C^{15}, \\
\Ext^1(\II_{Z'/V}(-4), \II_{Z'/V}(-4)) = \C^7.
\end{align*}
The moduli space of objects $\II_{Z'/V}$ is irreducible of dimension $5$. Lemma \ref{lem:ext_estimate} implies that all strictly semistable objects at the largest wall are smooth points on either $M_4$ or $M_5 = \Hilb^{4t}(\P^3)$. Therefore, we can again use Lemma \ref{lem:new_components} to see that $\Hilb^{4t}(\P^3)$ has exactly two components birational to $M_4^1$ and $M_4^2$. Moreover, this argument shows that the ideals that destabilize at the largest wall cannot lie on the intersection of the two components and we have $M_5^1 = M_4^1$.
\end{proof}
\end{thm}

We denote the exceptional divisor of the first blow up of the main component by $E_1$ and the exceptional divisor of the second blow up by $E_2$. We finish this section by describing which curves actually lie in $E_1$ and $E_2$.

\begin{prop}
\label{prop:E1_E2_description}
The general point in $E_1$ parametrizes subschemes of $\P^3$ that are the union of a plane cubic and an incident line. The general point in $E_2$ parametrizes subschemes of $\P^3$ that are plane quartics with two nodes and two embedded points at such nodes.
\begin{proof}
By Corollary \ref{cor:allwalls}, any ideal sheaf $I_C$ of a scheme in $E_1$ fits into an exact sequence of the form $0 \to \II_L(-1) \to \II_C \to \OO_V(-3) \to 0$, where $L \subset \P^3$ is a line and $V \subset \P^3$ is a plane. By Proposition \ref{prop:second_moduli} the reverse holds, i.e. all ideal sheaves fitting into this sequence correspond to curves in $E_1$. For the general element in $E_1$ the line $L$ is not contained $V$. Then the above sequence implies that $C \subset L \cup V$. If $C \subset V$, then there would be a morphism $\OO(-1) \to \II_C$ destabilizing the curve earlier, a contradiction. Thus, $L$ is an irreducible component of $C$ and another component of degree $3$ lies in $V$.

By Theorem \ref{thm:last_wall}, the last two walls do not modify the main component. Therefore, Corollary \ref{cor:allwalls} implies that any ideal sheaf $I_C$ of a scheme in $E_2$ fits into an exact sequence of the form $0 \to \II_Z(-1) \to \II_C \to \OO_V(-4) \to 0$, where $Z \subset \P^3$ is a zero dimensional subscheme of length $2$ and $V \subset \P^3$ is a plane. This implies that $C$ is plane quartic curve plus two points. The two points have to be embedded, since otherwise the curve cannot be smoothened. Moreover, a classical result by Hironaka \cite[p. 360]{Hir58} implies that the two embedded points must occur at singularities of the plane quartic.
\end{proof}
\end{prop}

\section{Effective divisors of the Principal Component $\Hilb_1^{4t}$}

In this section we compute the cone of effective divisors $\Eff(\Hilb_1^{4t})$, where $\Hilb_1^{4t}$ denotes the principal component of the Hilbert scheme $\Hilb^{4t}(\P^3)$.
By Theorem B, there is an isomorphism $\Pic(\Hilb_1^{4t})\cong \Z^3$, with generators $H$, $E_1$ and $E_2$. Here, $H$ denotes the pullback of the class $\sigma_1\in A^1(\G(1,9))$, whereas $E_1$ is the exceptional divisor of the first blow up and $E_2$ is the exceptional divisor of the second blow up. Set-theoretically, $E_1$ is the closure, in $\Hilb_1^{4t}$, of the locus parametrizing subschemes of $\P^3$ that consist of a smooth plane cubic with an incident line. Moreover, $E_2$ is the closure, in $\Hilb_1^{4t}$, of the locus parametrizing plane quartics with two nodes and two embedded points at such nodes.

As a consequence of Theorem B, we also have that $\Pic(\Hilb_1^{4t}) \otimes \Q \cong N^1(\Hilb_1^{4t})\otimes \Q$, where $N^1(\Hilb_1^{4t})\otimes \Q$ denotes the N\'eron-Severi group of Cartier divisors with rational coefficients up to numerical equivalence.

In order to describe the cone of effective divisors $\Eff(\Hilb_1^{4t})$, we need an additional divisor $\Delta$ defined as the closure of the locus of irreducible nodal elliptic quartics. 

\begin{thm:cones}
The cone of effective divisors of $\Hilb_1^{4t}$ is generated by $\Eff(\Hilb_1^{4t})=\langle\Delta, E_1,E_2\rangle$.
\end{thm:cones}

The strategy of the proof is to construct a dual basis of curves to $\Delta$, $E_1$, and $E_2$ in $N_1(\Hilb_1^{4t})$, the space of 1-cycles up to numerical equivalence. Since the closure of the convex cone of movable curves is dual to the effective cone, we will then observe that these curves are movable; which allows us to conclude the proof. In our context, a curve in a smooth algebraic variety $X$ is called \emph{movable}, if it lies in a family that covers a dense open subset of $X$. We refer the reader to \cite{BDPP13} for a detailed exposition on movable curves. 

Before proceeding with the proof, we will define and describe some properties of our movable curves. Let $Q\subset \P^3$ be a a fixed smooth quadric. Then, the curve $\gamma_1$ is a general pencil in $|\mathcal{O}_Q(2)|$. This curve is movable because a generic curve parametrized by $\Hilb_1^{4t}$ is the transversal intersection of two quadric hypersurfaces $Q_1,Q_2$ where we can assume one of these quadrics is smooth. Moreover, by construction $\gamma_1 \cdot E_1=\gamma_1 \cdot E_2=0$.

On the other hand, the intersection $\gamma_1 \cdot \Delta=12$ holds. This follows from the fact that the parameter space of plane curves of degree $d$ in $\P^2$ contains a divisor of degree $3(d-1)^2$ of singular curves (see \cite[Ch 13.D]{GKZ08}). The following geometric argument is self contained.

The base locus of a general pencil in $|\mathcal{O}_Q(2)|$, where $Q$ stands for a smooth quadric, consists of $8$ points. This means that the total space of this pencil $\mathcal{X}$ is the blowup of $Q$ on these $8$ points, and this implies that its topological Euler characteristic $\chi_{\tiny{top}}(\mathcal{X})=12$. Observe that the pencil $\mathcal{X}$ is not a fibration over $\P^1$ due to the presence of singular fibers: if $\mathcal{X}$ were a fibration over $\P^1$, then the topological Euler characteristic $\chi_{\tiny{top}}(\mathcal{X})$ would be zero. This means that we may count the singular fibers of $\mathcal{X}$ (which are the singular members of the pencil), by computing the topological Euler characteristic $\chi_{\tiny{top}}(\mathcal{X})$. Since we are considering a general pencil, Bertini's Theorem guarantees that the singular fibers of $\mathcal{X}$ are all nodal curves.

We now define two more curves $\gamma_2$ and $\gamma_3$. Let $\Lambda_1$ and $\Lambda_2$ be two $3$-planes in $\P^7$. Let $s: \P^3 \times \P^1 \to \P^7$ be the Segre embedding, and for any $t \in \P^1$ we write $s_t: \P^3 \to \P^7$ for the restriction of $s$ to $\P^3 \times \{ t \}$.
We have a projection $\pi: \P^7 \backslash \Lambda_1 \to \Lambda_2$. To summarize, we have the following diagram of maps with vertical projections
\[
\xymatrix{ 
\P^3 \times \P^1 \ar[r]^{s \times \id} \ar[d] & \P^7 \times \P^1 \ar@{-->}[r]^{\pi \times \id} \ar[d] & \Lambda_2 \times \P^1 \ar[d] \\
\P^3 \ar[r]^{s_t} & \P^7 \ar@{-->}[r]^{\pi} & \Lambda_2 \cong \P^3.
}
\]
Observe that both $s_t$ and $\pi$ are linear maps.

\begin{lem}
\label{lem:isomorphic_projection}
Let $t \in \P^1$, and let $\Lambda_2$ be general. If $\Lambda_1 \cap s_t(\P^3) = \emptyset$, then $\pi \circ s_t$ is an isomorphism. If $\Lambda_1 \cap s_t(\P^3)$ is a point, then the image of $\pi \circ s_t$ is a plane in $\Lambda_2$.
\begin{proof}
The image of $\pi \circ s_t$ is the intersection of $\Lambda_2$ with the linear subspace generated by $\Lambda_1$ and $s_t(\P^3)$.
\end{proof}
\end{lem}

The image of the Segre embedding $s(\P^3 \times \P^1)$ has degree four. Hence, $\Lambda_1$ can be chosen general such that it intersects the Segre embedding in exactly four points. If we also choose $\Lambda_2$ general, then by Lemma \ref{lem:isomorphic_projection}, we have that $\pi \circ s_t: \P^3 \to \Lambda_2 \cong \P^3$ is an isomorphism except for four values.

\begin{defn}
Let $E$ be a smooth elliptic quartic in $\P^3$. Let $\Lambda_2$ be a general $3$-plane in $\P^7$.
\begin{enumerate}
\item Let $\Lambda_1$ be another general $3$-plane in $\P^7$. Then $\gamma_2$ is the image $(\pi \times \id) \circ (s \times \id) (E \times \P^1)$. It is a flat family of smooth curves isomorphic to $E$ everywhere, except for four special fibers.
\item Consider four general points in $s(E \times \P^1)$ and let $\Lambda'_1$ be the unique $3$-plane generated by them. Then $\gamma_3$ is the image $(\pi \times \id) \circ (s \times \id) (E \times \P^1)$. It is a flat family of smooth curves isomorphic to $E$ everywhere except for four special fibers.
\end{enumerate}
\end{defn}

\begin{lem}
\label{lem:singular_fibers_curves}
The four singular fibers for $\gamma_2$ are plane quartic curves with only two nodes and embedded points at them. For $\gamma_3$ these four fibers are smooth plane cubic curves together with an incident line. Both $\gamma_2$ and $\gamma_3$ are movable.
\begin{proof}
Let $t \in \P^1$ correspond to one of the four singular fibers of $\gamma_2$. Since $\Lambda_1$ is chosen general it will not intersect $s(E \times \P^1)$. Therefore, Lemma \ref{lem:isomorphic_projection} implies that the image $\pi(s_t(E))$ is a plane curve. Since $\pi \circ s_t$ is defined on all of $E$, the set-theoretic support of $\gamma_2$ at $t$ is a plane curve of degree four with $2$ nodes and no other singularities. Hence, we get a plane quartic with two embedded points at the only $2$ nodes.

Let $t \in \P^1$ correspond to one of the four singular fibers $\gamma_3$. By definition the intersection of $\Lambda'_1$ with $E \times \P^1$ contains four points one of which is of the form $(x,t)$. Choose a plane $\P^2 \subset \Lambda'_1$ that does not intersect the Segre embedding $s(\P^3 \times \P^1)$ and a general $\P^4 \subset \P^7$. Then the projection of $s_t(\P^3)$ away from $\P^2$ onto $\P^4$ is the intersection of this $\P^4$ with the linear span of $s_t(\P^3)$ and $\P^2$ which is a $\P^6$. In particular, it is of dimension $3$, i.e. $E$ is projected isomorphically onto $\P^3 \subset \P^4$. Let $P \in \P^4$ be the image of $(x,t)$ via this projection. Then we project from this point onto a general $\Lambda_2 \subset \P^4$. The image is isomorphic to $E$. Hence, we get in $\Hilb_1^{4t}$ a smooth plane cubic together with an incident line.

Both curve classes $\gamma_2$ and $\gamma_3$ are movable. Indeed, every smooth curve parametrized in $\Hilb_1^{4t}$ is contained in some representative of $\gamma_2$ and $\gamma_3$ by varying the curve $E$ used to define them.
\end{proof}
\end{lem}

\begin{proof}[Proof of Theorem C]
Since $E_1$, $E_2$, and $\Delta$ are effective, we only need to show the containment $\Eff(\Hilb_1^{4t})\subset \langle E_1, E_2,\Delta\rangle$. Observe that this latter containment is equivalent to $\langle E_1,E_2,\Delta\rangle^{\vee} \subset \Eff(\Hilb_1^{4t})^{\vee}$ of dual cones. Since $\Eff(\Hilb_1^{4t})^{\vee}$ is the cone of movable curves, it suffices to exhibit that the dual cone $\langle E_1, E_2, \Delta\rangle^{\vee}$ is generated by movable curves. We already proved that $\gamma_1,\gamma_2,\gamma_3$ are movable curves. We will

This means we are left to show they generate the dual cone $\langle E_1$, $E_2, \Delta\rangle^{\vee}$. It suffices to check that the following intersection numbers hold. Note that for our purposes it is enough to show that the intersections are zero or positive, and therefore, we will skip proving that the intersections are transversal.
\begin{equation*}
\begin{aligned}
 \gamma_1 \cdot E_1=0,\quad & \gamma_2 \cdot E_1=0, \quad & \gamma_3 \cdot E_1=4, \\
 \gamma_1 \cdot E_2=0, \quad & \gamma_2 \cdot E_2=4, \quad & \gamma_3 \cdot E_2 =0,\\
\gamma_1 \cdot \Delta=12,\quad & \gamma_2 \cdot \Delta=0, \quad & \gamma_3 \cdot \Delta=0. \\
\end{aligned}
\end{equation*}
The intersections with $E_1$ and $E_2$ follow directly from the definitions and Lemma \ref{lem:singular_fibers_curves}. The intersection numbers $\gamma_1 \cdot \Delta = 12$ is also discussed above. We are left to show $\gamma_2 \cdot \Delta=\gamma_3 \cdot \Delta=0$.

Suppose $\gamma_2 \cdot \Delta \neq 0$. Then there is a flat family $\pi: S \to Z$ for a smooth curve $Z$ such that for general $z \in Z$ the fiber $S_z$ is a nodal complete intersection in $\Delta$, and the special fiber $S_0$ is a curve in $\gamma_2 \cap E_2$. Therefore, $S_0$ is a plane quartic curve with exactly two nodes and simple embedded points at both nodes. The normalization $\tilde{S}$ smooths out the nodes in the general fibers by making them into $\P^1$. By \cite[Theorem III.7]{BEAU} this means $\tilde{S}$ is birational over $Z$ to $\P^1 \times Z$. We can resolve the rational map from $\P^1 \times Z$ to $S$ by successively blowing up points. That leads to a family $X \to Z$ factoring through $S \to Z$ such that every fiber is a union of rational curves $\P^1$. That means the special fiber $S_0$ is the set theoretic image of such a union of rational curves. Every $\P^1$ must map to the normalization of the reduced structure of $S_0$. But the normalization of the reduced structure of $S_0$ is a smooth curve of genus $1$, and $\P^1$ has no non trivial maps to an elliptic curve.

Suppose $\gamma_3 \cdot \Delta \neq 0$. Then there is a flat family $\pi: S \to Z$ for a smooth curve $Z$ such that for general $z \in Z$ the fiber $S_z$ is a nodal complete intersection in $\Delta$, and the special fiber $S_0$ is a curve in $\gamma_3 \cap E_1$. This means $S_0$ is the union of a smooth plane cubic with an incident line. With the exact same argument as for $\gamma_2$, we can create a family $X \to Z$ factoring through $S \to Z$ such that every fiber is a union of rational curves $\P^1$. As previously, the special fiber $S_0$ is the image of such a union of rational curves. Since there is no non-trivial map from $\P^1$ to any elliptic curve, they must all map to the incident line, a contradiction.
\end{proof}

\section*{Appendix A: Macaulay2 Code}

This appendix contains all Macaulay2 code used in Proposition \ref{prop:piene-schlessinger} and Lemma \ref{lem:ext-computations}.

{\tiny
\begin{multicols}{2}
\begin{lstlisting}
-------------------------------------
-- Computation for Proposition 3.2 --
-------------------------------------
needsPackage "VersalDeformations";
S = QQ[x,y,z,w];
F0 = matrix {{x^2,x*y,x*z^2,y^4}};
(F,R,G,C) = localHilbertScheme(F0,Verbose=>2);
T = ring first G;
sum G
-------------------------------
-- Computation for Lemma 4.3 --
-------------------------------
-- A = O_V(-3)
-- B = I_L(-1) L is not contained in V
-- C = I_L(-1) L is contained in V
X = Proj(QQ[x,y,z,w]);
A = (OO_X(0) / sheaf module ideal(x))**OO_X(-3);
B = (sheaf module ideal (y,z))**OO_X(-1);
C = (sheaf module ideal (x,y))**OO_X(-1);
Ext^1(B,A)
Ext^1(C,A)
Ext^1(A,B)
Ext^1(A,C)
Ext^1(A,A)
Ext^1(B,B)
Ext^1(C,C)
-------------------------------
-- A = O_V(-4)
-- B = I_Z(-1) Two separate points
--             outside V
-- C = I_Z(-1) Double point outside V
-- D = I_Z(-1) One point inside,
--             one point outside V
-- E = I_Z(-1) Two separate points
--             inside V
-- F = I_Z(-1) Double point scheme
--             theoretically in V
-- G = I_Z(-1) Double point set but
--             not scheme theoretically in V
X = Proj(QQ[x,y,z,w]);
A = (OO_X(0) / sheaf module ideal(x))**OO_X(-4);
B = (sheaf module ideal (y*(x-y),z,w))**OO_X(-1);
C = (sheaf module ideal (y^2,z,w))**OO_X(-1);
D = (sheaf module ideal (x*y,z,w))**OO_X(-1);
E = (sheaf module ideal (x,y*z,w))**OO_X(-1);
F = (sheaf module ideal (x,y,z^2))**OO_X(-1);
G = (sheaf module ideal (x^2,y,z))**OO_X(-1);
Ext^1(A,B)
Ext^1(A,C)
Ext^1(A,D)
Ext^1(A,E)
Ext^1(A,F)
Ext^1(A,G)
Ext^1(B,A)
Ext^1(C,A)
Ext^1(D,A)
Ext^1(E,A)
Ext^1(F,A)
Ext^1(G,A)
Ext^1(A,A)
Ext^1(B,B)
Ext^1(C,C)
-------------------------------
-- A = I_{Q/V}(-4)
-- B = I_P(-1) P \notin V
-- C = I_P(-1) P \in V, P \neq Q
-- D = I_P(-1) P = Q
X = Proj(QQ[x,y,z,w]);
A = (sheaf module ideal (x,y,z)
    / sheaf module ideal(x))**OO_X(-4);
B = (sheaf module ideal (y,z,w))**OO_X(-1);
C = (sheaf module ideal (x,y,w))**OO_X(-1);
D = (sheaf module ideal (x,y,z))**OO_X(-1);
Ext^1(A,B)
Ext^1(A,C)
Ext^1(A,D)
Ext^1(B,A)
Ext^1(C,A)
Ext^1(D,A)
Ext^1(A,A)
Ext^1(B,B)
Ext^1(D,D)
-------------------------------
-- A = O(-1)
-- B = I_{Z'/V}(-4) Two separate points
-- C = I_{Z'/V}(-4) Double point
X = Proj(QQ[x,y,z,w]);
A = OO_X(-1);
B = (sheaf module ideal (x,y,z^2)
    / sheaf module ideal(x))**OO_X(-4);
C = (sheaf module ideal (x,y,w*z)
    / sheaf module ideal(x))**OO_X(-4);
Ext^1(B,A)
Ext^1(C,A)
Ext^1(A,B)
Ext^1(A,C)
Ext^1(B,B)
Ext^1(C,C)
-------------------------------
-- A = I_Z(-1), Z \subset V double point at P
-- B = O_V(-4)
-- C = I_{P/V}(-4)
-- D = I_P(-1)
X = Proj(QQ[x,y,z,w]);
A = (sheaf module ideal(x,y,z^2))**OO_X(-1);
B = (OO_X(0) / sheaf module ideal(x))**OO_X(-4);
C = (sheaf module ideal(x,y,z)
    / sheaf module ideal(x))**OO_X(-4);
D = (sheaf module ideal(x,y,z))**OO_X(-1);
Ext^1(A,C)
Ext^1(B,C)
Ext^1(A,D)
Ext^1(B,D)
\end{lstlisting}
\end{multicols}
}

\end{document}